\documentclass[a4paper,reqno]{amsart} 

\usepackage{amssymb, mathtools, enumerate}

\newcommand \la{\lambda}
\newcommand \br{\mathbb{R}}

\newcommand \Span{\operatorname{Span}}
\newcommand \<{\langle}
\renewcommand \>{\rangle}
\newcommand \ip{\< \cdot, \cdot \>}
\newcommand \mS{\mathcal{S}}
\newcommand \cR{\mathcal{R}}
\newcommand \Rn{\mathbb{R}^n}
\newcommand \Cn{\mathbb{C}^n}
\newcommand \Tr{\operatorname{Tr}}
\newcommand \cX{\mathcal{X}}
\newcommand \cQ{\mathcal{Q}}

\theoremstyle{plane}

\newtheorem*{theorem*}{Theorem}
\newtheorem*{corollary*}{Corollary}
\newtheorem*{conj*}{Conjecture}
\newtheorem{lemma}{Lemma}
\newtheorem{proposition}{Proposition}
\newtheorem*{prop*}{Proposition}

\theoremstyle{definition}

\newtheorem*{definition*}{Definition}

\theoremstyle{remark}
\newtheorem*{remark*}{Remark}

\begin{document}

\title{$H$-contact unit tangent sphere bundles of Riemannian manifolds} 

\author{Yuri Nikolayevsky}
\address{Department of Mathematics and Statistics, La Trobe University, Melbourne, Victoria, 3086, Australia}
\email{y.nikolayevsky@latrobe.edu.au}
\thanks{Y. Nikolayevsky was supported by ARC Discovery grant DP130103485 and by Korea Institute for Advanced Study} 

\author{Jeong Hyeong Park}
\address{Department of Mathematics, Sungkyunkwan University, Suwon, 16419, Korea} 
\email{parkj@skku.edu}
\thanks{J.H. Park was supported by the Basic Science Research Program through the National Research Foundation of Korea(NRF) funded by the Ministry of Education (2014053413)}

\subjclass[2010]{Primary 53C25, 53D10; Secondary 53B20}
\keywords{tangent sphere bundle, $H$-contact manifold, $2$-stein}


\date{}

\begin{abstract}
A contact metric manifold is said to be \emph{$H$-contact}, if the characteristic vector field is harmonic. We prove that the unit tangent bundle of a Riemannian manifold $M$ equipped with the standard contact metric structure is $H$-contact if and only if $M$ is $2$-stein.
\end{abstract}

\maketitle

\section{Introduction}
\label{s:intro}

Let $(\tilde M, \tilde g)$ be a compact, orientable Riemannian manifold. The \emph{energy} $E(V)$ of a unit vector field $V$ is defined as the energy of the
corresponding map between $(\tilde M, \tilde g)$ and its tangent sphere bundle equipped with the Sasaki metric:
\begin{equation*}
E(V)=\frac12 \int_M |dV|^2 dv_{\tilde g} =\frac{m}{2} \mathrm{Vol}(\tilde M, \tilde g) +\frac12 \int_M |\nabla V|^2 dv_{\tilde g},
\end{equation*}
where $m = \dim M$ \cite{Wood}. Wood defined a unit vector field to be \emph{harmonic} if it is a critical point for the energy functional $E$ in the set of all unit vector fields of $\tilde M$, and then by considering the first variation obtained a local condition for harmonicity of a vector field \cite{Wood}.

A contact metric manifold whose characteristic vector field $\xi$ is harmonic is called an \emph{$H$-contact manifold}. Perrone proved that a contact metric manifold is $H$-contact if and only if the characteristic vector field $\xi$ is an eigenvector of the Ricci operator \cite{Pe1}.

A substantial progress has been achieved in the study of this construction in the case when the contact metric manifold is the unit tangent sphere bundle of a Riemannian manifold $(M,g)$ equipped with the Sasaki metric and the standard contact structure. Boeckx and Vanhecke \cite{BV2} showed that the unit tangent sphere bundle of a $2$-dimensional or a $3$-dimensional Riemannian manifold is $H$-contact if and only if the base manifold $(M,g)$ has constant sectional curvature. In \cite{CP}, Calvaruso and Perrone proved that the same is true under assumption that $(M,g)$ is conformally flat, and also obtained a local characterisation of such manifolds $(M,g)$ in the general case (see Proposition~\ref{p:cp} in Section~\ref{s:proof} below). That characterisation was used in \cite{CPS3} to show that a Riemannian manifold whose unit tangent sphere bundle is $H$-contact has constant scalar curvature, constant norm of the Ricci tensor and constant norm of the curvature tensor (the latter is true when $\dim M \ne 4$; there is a counterexample in dimension $4$). It was further established that a Riemannian manifold whose unit tangent sphere bundle is $H$-contact is $2$-stein, provided that either $(M,g)$ is Einstein \cite{CPS2} or $\dim M = 4$ \cite{CPS4}.

Our main result is as follows.
\begin{theorem*}
Let $(M,g)$ be a Riemannian manifold. The unit tangent sphere bundle $T_1M$ equipped with the standard contact metric structure is $H$-contact if and only if $(M,g)$ is $2$-stein. 
\end{theorem*}

Recall that an $n$-dimensional Riemannian manifold $(M,g)$ is said to be \emph{2-stein} if there exist two functions $f_1, f_2: M \to \br$ such that for every $p \in M$ and every vector $X$ tangent to $M$ at $p$ we have
\begin{equation*}
    \Tr R_X = f_1(p) \|X\|^2, \qquad \Tr (R_X^2) = f_2(p) \|X\|^4,
\end{equation*}
where $R_X$ is the Jacobi operator at $p$ \cite[p.~47]{CGW}. In particular, any 2-stein manifold is Einstein. By Schur's Theorem, the function $f_1$ is constant when $n \ge 3$, and by \cite[\S 6.57, 6.61]{Bes}, the function $f_2$ is constant when $n \ge 5$. $2$-stein manifolds show up in many questions in Riemannian geometry, to name one, in the theory of harmonic spaces \cite[\S 6.46]{Bes}. In dimension $n=4$, the structure of the curvature tensor of a $2$-stein manifold is well known \cite[Lemma~7]{SV}; any four-dimensional $2$-stein manifold is pointwise Osserman \cite[\S 6.73]{Bes}, hence self-dual (up to the choice of orientation) \cite{BG}, and as such, is either hyperk\"{a}hler or quaternionic K\"{a}hler depending on the scalar curvature. In dimension $n=5$ one can be much more specific (perhaps the true reason for that is the fact that $f_2$ is constant): any $2$-stein manifold either has constant curvature, or up to scaling, is locally isometric either to the symmetric space $\mathrm{SU}(3)/\mathrm{SO}(3)$ or to its non-compact dual $\mathrm{SL}(3)/\mathrm{SO}(3)$ \cite[Proposition~1]{N}. The classification of $2$-stein spaces is known in the locally symmetric case \cite{CGW} and for some other classes of manifolds.

\smallskip 

In Section~\ref{s:prelim} we give necessary background on contact geometry and the Sasaki metric. The proof of the Theorem is given in Section~\ref{s:proof}; the core of the proof is a purely algebraic Proposition~\ref{p:act}.

All the objects (manifolds, metrics, vector fields, forms) in this paper are assumed to be of class $C^{\infty}$.

%



\section{Standard contact metric structure on the unit tangent sphere bundle}\label{s:prelim}

We start with some preliminaries on a contact metric manifolds (the reader is referred  to \cite{Blair1} for more details). A $(2n-1)$-dimensional manifold $\bar M$ is said to be \emph{contact} if it admits a global $1$-form $\eta$ such that $\eta\wedge(d\eta)^{n-1}\neq 0$ everywhere on $\bar M$, where the exponent denotes the $(n-1)$-st exterior power. We call such an $\eta$ a \emph{contact form} on $\bar M$. Given a contact form $\eta$, there exists a unique vector field $\xi$, the \emph{characteristic vector field}, satisfying $\eta(\xi)=1$ and $d\eta(\xi, X)=0$, for any vector field $X$ on $\bar M$. A Riemannian metric $\bar g$ on $\bar M$ is said to be \emph{an associated metric} to a contact form $\eta$ if there exists a $(1,1)$-tensor field $\phi$ satisfying
\begin{equation} \label{eq:asso}
        \eta(X)=\bar g(X,\xi),\quad d\eta(X, Y)=\bar g(X,\phi Y), \quad \phi^2 X=- X+\eta(X)\xi,
\end{equation}
for any vector fields $X$ and $Y$ on $\bar M$. A Riemannian manifold $\bar M$ equipped with structure tensors $(\bar g,\phi,\xi,\eta)$ satisfying
 \eqref{eq:asso} is called a \emph{contact metric manifold}.

\smallskip

Let $(M,g)$ be a Riemannian manifold, with the Levi-Civita connection $\nabla$. The tangent bundle $TM$ of $(M,g)$ consists of pairs $(p,u)$, where $p \in M$ and $u$ a tangent vector to $M$ at $p$. The mapping $\pi: TM \rightarrow M, \; \pi(p,u)=p$, is the natural projection from $TM$ onto $M$.



For a vector field $X$ on $M$, its \emph{vertical lift} $X^v$ is the unique vector field on $TM$ defined by $X^v \omega =\omega(X) \circ \pi$, where $\omega$ is a $1$-form on $M$, and the \emph{horizontal lift} $X^h$  is the unique vector field on $TM$ is defined by $X^h \omega =\nabla_X\omega$. Any vector tangent to $TM$ at $(p,u)$ can be uniquely represented as $X^h+Y^v$ for some vectors $X$ and $Y$ tangent to $M$ at $p$. The tangent bundle $TM$ can be endowed in a natural way with a Riemannian metric $g_S$, the \emph{Sasaki metric} defined as follows:
\begin{equation*}
  g_S(X^h,Y^h) = g_S(X^v,Y^v) = g(X,Y) \circ \pi, \quad g_S(X^h,Y^v)=0,
\end{equation*}
for any vector fields $X$ and $Y$ on $M$ (for more details on the Sasaki metric we refer the reader to the survey papers \cite{BV1, BY}). The Sasaki metric $g_S$ is Hermitian for the almost complex structure $J$ defined by $JX^h=X^v$ and $JX^v=-X^h$.

The unit tangent sphere bundle is the hypersurface of $TM$ given by $g_p(u,u)=1$. The unit normal vector field $N = u^v$ to $T_1M$ is the vertical lift of $u$ to $(p,u) \in T_1 M$.

We now define the standard contact metric structure of the unit tangent sphere bundle $T_1M$ of a Riemannian manifold $(M,g)$. The metric $g'$ on $T_1M$ is
induced from the Sasaki metric $g_S$ on $TM$. Using the almost complex structure $J$ on $TM$, we define the unit vector field $\xi'$, the $1$-form $\eta'$ and the $(1,1)$ tensor field $\phi'$ on $T_1M$ by $\xi'=-JN$ and $\phi' = J - \eta' \otimes N$. Since $g'(\bar X,\phi' \bar Y) = 2d\eta'(\bar X,\bar Y)$,
the quadruple $(g',\phi',\xi',\eta')$ is not a contact metric structure. By rescaling
\begin{equation*}
\xi=2\xi',\qquad \eta=\frac{1}{2}\,\eta',\qquad \phi=\phi',\qquad \bar g=\frac{1}{4}\, g',
\end{equation*}
we get \emph{the standard contact metric structure} $(\bar g, \phi, \xi,\eta)$ on $T_1 M$.

\section{Proof of the Theorem}
\label{s:proof}

Let $(M,g)$ be an $n$-dimensional Riemannian manifold, with $\nabla, R$ and $\rho$ the Levi-Civita connection, the Riemann curvature tensor and the Ricci tensor respectively.

Suppose that the unit tangent sphere bundle $T_1M$ equipped with the standard contact metric structure is $H$-contact.

By \cite[Proposition~2]{BV2} and \cite{CPS4} we can assume that $n \ge 5$. By \cite[Proposition~3.1]{CP} we have the following.

\begin{proposition} \label{p:cp}
The unit tangent sphere bundle $T_1M$ of a Riemannian manifold $(M,g)$ is $H$-contact with respect to the standard contact metric structure $(\bar g, \phi, \xi, \eta )$ if and only if the following two conditions are satisfied.
\begin{enumerate}[{\rm (a)}]
  \item \label{it:cp1}
  The Ricci tensor is \emph{Codazzi}, that is, for arbitrary vector fields $X,Y$ and $Z$ on $M$ we have
  \begin{equation}\label{hc1}
    \nabla_X \rho(Y,Z)=\nabla_Y \rho(X,Z).
  \end{equation}

  \item \label{it:cp2}
  For unit, orthogonal vector fields $X, Y$ on $M$, we have
  \begin{equation}\label{hc2}
    \sum_{i=1}^n g(R(X,e_i)X, R(X,e_i)Y) = 2\rho(X,Y),
  \end{equation}
  where $\{e_i\}_{i=1}^n$ is a (local) orthonormal frame on $M$.
\end{enumerate}
\end{proposition}

By the result of \cite{CPS2}, it suffices to prove that $(M,g)$ is Einstein. Seeking a contradiction assume that it is not. Let $p \in M$ be a point having the maximal number, $N$, of pairwise non-equal Ricci eigenvalues. By our assumption $N \ge 2$, and by construction, in a neigbourhood of $p$, the Ricci tensor $\rho$ has $N$ smooth eigendistributions, of constant dimensions. Let $E_1, \dots, E_N$ be the eigenspaces of $\rho$ at $p$. By \cite[Theorem~1]{DS} the fact that $\rho$ is a Codazzi tensor implies that for $1 \le \la, \mu, \nu \le N$,
\begin{equation}\label{eq:codazzi}
R(E_\la, E_\mu) E_\nu=0, \; \text{ when }\nu \notin \{\la, \mu\}.
\end{equation}

We now consider condition~\eqref{it:cp2}. Given a Euclidean space $(\Rn, \ip)$ (with the inner product $\ip$ and the norm $\| \cdot\|$), we define an \emph{algebraic curvature tensor} to be a $(4,0)$ tensor having the same algebraic symmetries as the curvature tensor of a Riemannian manifold. Given an algebraic curvature tensor $\cR$, we can define the corresponding Ricci tensor, and for any $X \in \Rn$, the Jacobi operator $\cR_X$.

Take $(\Rn, \ip)$ to be the tangent space to $M$ at the point $p$ and define the algebraic curvature tensor $\cR$ on $(\Rn, \ip)$ by
\begin{equation}\label{eq:defcR}
\cR(X,Y,Z,W)=R(X,Y,Z,W)-2(\<X,Z\>\<Y,W\>-\<X,W\>\<Y,Z\>).
\end{equation}
Note that $\cR$ is obtained from $R$ by shifting by an algebraic curvature tensor of constant curvature, and so equation \eqref{eq:codazzi} with $R$ replaced by $\cR$ is still satisfied. What is more, we have the following fact.

\begin{lemma}\label{l:shift}
Let $R$ be an algebraic curvature tensor in $(\Rn, \ip)$. The following two conditions are equivalent.
\begin{enumerate}[{\rm (a)}]
  \item
  The algebraic curvature tensor $R$ satisfies \eqref{hc2}, for any unit, orthogonal vectors $X, Y \in \Rn$.

  \item
  There exists $H \ge 0$ such that the algebraic curvature tensor $\cR$ defined by \eqref{eq:defcR} satisfies the equation
  \begin{equation*}
    \Tr (\cR_X^2)= H \|X\|^4,
  \end{equation*}
  for all $X \in \Rn$.
\end{enumerate}
\end{lemma}
\begin{proof}
Let $\{e_i\}_{i=1}^n$ be an orthonormal basis for $(\Rn,\ip)$ and let $X, Y \in \Rn$ be unit, orthonormal vectors. Then \eqref{hc2} is equivalent to
  \begin{align*}
    \sum_{i=1}^n & \<\cR(X,e_i)X, \cR(X,e_i)Y\> = \sum_{i=1}^n \<R(X,e_i)X-2(e_i-\<X,e_i\>X), R(X,e_i)Y+2\<Y,e_i\>X\> \\
    &= 2\rho(X,Y)-2\sum_{i=1}^n \<R(X,e_i)Y, e_i-\<X,e_i\>X\> -4\sum_{i=1}^n \<e_i-\<X,e_i\>X, \<Y,e_i\>X\>= 0,
  \end{align*}
and so $\sum_{i=1}^n \<\cR(X,e_i)X, \cR(X,e_i)Y\> =0$, for any two orthogonal vectors $X$ and $Y$ (not necessarily unit). Consider the function $F:\Rn\setminus\{0\} \to \br$ defined by $F(X) = \|X\|^{-4}\Tr (\cR_X^2) = \|X\|^{-4} \sum_{i=1}^n \<\cR(X,e_i)X, \cR(X,e_i)X\>$. Then $(\partial_X F)(X)=0$, as $F$ is $0$-homogeneous in $X$, and $(\partial_Y F)(X)=0$ for $Y \perp X$, by the equation above. It follows that $F$ is a (non-negative) constant. Conversely, if $\Tr (\cR_X^2)= H \|X\|^4$, then the equation $\sum_{i=1}^n \<\cR(X,e_i)X, \cR(X,e_i)Y\> =0$ for $X \perp Y$ follows by polarisation.
\end{proof}

The proof of the Theorem is now concluded by the following purely algebraic fact.
\begin{proposition}\label{p:act}
Let $\cR$ be an algebraic curvature tensor in $(\Rn, \ip), \; n \ge 5$. Suppose that
\begin{enumerate}[{\rm (a)}]
  \item \label{it:act1}
  There exists $H \ge 0$ such that for all $X \in \Rn$ we have
  \begin{equation}\label{eq:RXsq}
    \Tr (\cR_X^2)= H \|X\|^4.
  \end{equation}

  \item \label{it:act2}
  There exists a direct orthogonal decomposition $\Rn=W_1 \oplus W_2$, with $\dim W_i > 0$, such that
    \begin{equation}\label{eq:codazzicR}
        \cR(W_1, W_1) W_2 = \cR(W_2, W_2) W_1 = 0.
    \end{equation}
\end{enumerate}
Then $\cR$ has constant curvature.
\end{proposition}

To see that Proposition~\ref{p:act} indeed implies the Theorem we note that for the algebraic curvature tensor $\cR$ defined by \eqref{eq:defcR}, condition \eqref{it:act2} is satisfied by \eqref{eq:codazzi} if we take $W_1=E_1, \; W_2=\oplus_{\la=2}^N E_\la$, and condition \eqref{it:act1}, by Lemma~\ref{l:shift}. Then $\cR$ has constant curvature, which by \eqref{eq:defcR} implies that $R$ has constant curvature at $p$ contradicting the fact that the Ricci tensor $\rho$ at $p$ has $N \ge 2$ pairwise distinct eigenvalues.

\begin{proof}[Proof of Proposition~\ref{p:act}]
We first give a brief, informal sketch of the proof. Complexifying everything we get that condition \eqref{eq:RXsq} is valid for any $X \in \Cn$. Taking a particular $X \in \Cn$ we can symmetrise \eqref{eq:RXsq} by the product of two symmetric groups: the permutations of the coordinates of $X$ in $W_1$ and in $W_2$ respectively. Then the left-hand side of \eqref{eq:RXsq} becomes a quadratic form in the components of $\cR$ with coefficients depending on elementary symmetric functions $\sigma_h$ of the corresponding sets of coordinates (note that the left-hand side of \eqref{eq:RXsq} has degree four in the coordinates of $X$, so we will have only $\sigma_h$ with $h \le 4$). We then choose a set of vectors $X^\alpha \in \Cn$ such that $\sum_\alpha \|X^\alpha\|^4=0$ and take the sum of the above equations by $\alpha$. Then the right-hand side of the resulting equation becomes zero, and the left-hand side, if we choose our vectors $X^\alpha \in \Cn$ ``in the correct way", becomes not only real, but a \emph{positive semidefinite} quadratic form in the components of $\cR$. This will give us a set of linear equations in the components of $\cR$ which will imply that $\cR$ has constant curvature.

\smallskip

Beginning in earnest, we denote $d_1 = \dim W_1 \ge 1, \; d_2=\dim W_2 \ge 1$, assume that $d_1 \le d_2$ (recall that $d_1+d_2 = n \ge 5$), and adopt the following index convention: $1 \le i,j,k,l \le d_1; \; d_1+1 \le a,b,c,d \le d_1+d_2$. In all summations below, the indices run over the corresponding ranges.

Choose an orthonormal basis $\{e_i,e_a\}$ for $(\Rn, \ip)$ such that $W_1=\Span_i(e_i), \; W_2=\Span_a(e_a)$. Then by \eqref{eq:codazzicR} we have
\begin{equation}\label{eq:cR1}
    \cR_{ijka}=\cR_{ijab}=\cR_{iabc}=0, \quad   \cR_{iajb}=\cR_{ibja},
\end{equation}
for all the values of the subscripts, where the latter equations follows from the first Bianchi identity.

The algebraic equation \eqref{eq:RXsq} holds in the complexification $\Cn$ of $\Rn$ (with the inner product extended from that on $\Rn$ by complex linearity). Let $X=\sum_i x_i e_i + \sum_a y_a e_a$. Then using \eqref{eq:cR1} we get
\begin{align*}
    (\cR_X)_{ij} &= \sum_{kl} x_k x_l \cR_{kilj} + \sum_{cd} y_c y_d \cR_{cidj}, \; \;
    (\cR_X)_{ab} = \sum_{kl} x_k x_l \cR_{kalb} + \sum_{cd} y_c y_d \cR_{cadb}, \; \;
    (\cR_X)_{ia} = \sum_{kc} x_k y_c \cR_{cika},
\end{align*}
and so
\begin{equation}\label{eq:RXsqexpand}
\begin{split}
    \Tr (\cR_X^2) &= \sum_i x_i^4 p_{1|i} + \sum_{i\ne j} x_i^3 x_j p_{2|ij} + \sum_{i\ne j} x_i^2 x_j^2 p_{3|ij} + {\sum_{ijk}}' x_i^2 x_j x_k p_{4|ijk} + {\sum_{ijkl}}' x_i x_j x_k x_l p_{5|ijkl} \\
    & + \sum_a y_a^4 q_{1|a} + \sum_{a\ne b} y_a^3 y_b q_{2|ab} + \sum_{a\ne b} y_a^2 y_b^2 q_{3|ab} + {\sum_{abc}}' y_a^2 y_b y_c q_{4|abc} + {\sum_{abcd}}' y_a y_b y_c y_d q_{5|abcd} \\
    & + 2 \sum_{ia} x_i^2 y_a^2 s_{1|ia} + 2 \sum_{a\ne b;i} x_i y_a y_b s_{2|iab} + 2 \sum_{i\ne j;a} x_i x_j y_a s_{3|ija} + 2 \sum_{a \ne b;i\ne j} x_i x_j y_a y_b s_{4|ijab},
\end{split}
\end{equation}
where here and below we denote $\sum'$ the summation by all the pairwise nonequal values of the subscripts in the respective ranges, and where $p_{1|i}, p_{2|ij}, \dots, s_{4|ijab}$ are quadratic forms in the components of $\cR$. In particular, 
\begin{align*}
    p_{1|i} &= \sum_{kl} \cR_{ikil}^2 + \sum_{ab} \cR_{iaib}^2,\\
    p_{3|ij} &= \sum_{kl} \cR_{ikil}\cR_{jkjl} + \sum_{kl} \cR_{ikjl}^2 + \sum_{kl} \cR_{ikjl}\cR_{jkil} + \sum_{ab} \cR_{iaib}\cR_{jajb} + 2 \sum_{ab} \cR_{iajb}^2,\\
    q_{1|a} &= \sum_{cd} \cR_{acad}^2 + \sum_{kl} \cR_{akal}^2,\\
    q_{3|ab} &= \sum_{cd} \cR_{acad}\cR_{bcbd} + \sum_{cd} \cR_{acbd}^2 + \sum_{cd} \cR_{acbd}\cR_{bcad} + \sum_{kl} \cR_{akal}\cR_{bkbl} + 2 \sum_{kl} \cR_{akbl}^2,\\
    s_{1|ia} & = \sum_{kl} \cR_{ikil} \cR_{akal} + \sum_{cd} \cR_{icid} \cR_{acad} + \sum_{kc} \cR_{akic}^2,
\end{align*}
where in the right-most terms of $p_{3|ij}, q_{3|ab}$ and $s_{1|ia}$ we used the last equation of \eqref{eq:cR1}. We now take the sum of the expressions for $\Tr (\cR_X^2)$ given by \eqref{eq:RXsqexpand} by all the permutations of the $x_i$ and of the $y_a$. The resulting expression $\mS(X)$ depends on the elementary symmetric functions $\sigma_h(x), \sigma_h(y), \; h=1,2,3,4$, of the variables $x_i$ and $y_a$ respectively (rather than on these variables as such), where we denote $\sigma_1(x)=\sum_i x_i, \; \sigma_2(x)=\sum_{i \ne j} x_i x_j, \; \sigma_3(x)={\sum}'_{ijk} x_i x_j x_k, \; \sigma_4(x)={\sum}'_{ijkl} x_i x_j x_k x_l$, and similarly, for $\sigma_h(y)$.

Let $\mathcal{Z} \subset \Cn$ be the set of vectors $X=\sum_i x_i e_i + \sum_a y_a e_a$ such that no more than two of the $x_i$ and no more than two of the $y_a$ are nonzero. Assuming that the vector $X$ is chosen in $\mathcal{Z}$ we get $\sigma_3(x)=\sigma_4(x)=\sigma_3(y)=\sigma_4(y)=0$.

Performing the summation by all the permutations of the $x_i$ and of the $y_a$ in every term on the right-hand side of \eqref{eq:RXsqexpand} and then expressing the coefficients in terms of $\sigma_h(x), \sigma_h(y)$ we get
\begin{equation} \label{eq:Sx}
    \mS(X) = \sum_{h=1}^3 A_h(x) P_h + \sum_{h=1}^3 B_h(y) Q_h + \sum_{h=1}^4 C_h(x,y) S_h,
\end{equation}
with
\begin{equation}\label{eq:AhZ}
\begin{split}
    A_1(x) &=(d_1-1)!d_2! (\sigma_1^4(x)-4 \sigma_1^2(x) \sigma_2(x) + 2 \sigma_2^2(x)), \\
    A_2(x) &=(d_1-2)!d_2! (\sigma_1^2(x) \sigma_2(x) - 2 \sigma_2^2(x)), \\
    A_3(x) &=(d_1-2)!d_2! (2 \sigma_2^2(x)), \\
    C_1(x,y) &=(d_1-1)!(d_2-1)! 2(\sigma_1^2(x) - 2\sigma_2(x))(\sigma_1^2(y) - 2\sigma_2(y)), \\
    C_2(x,y) &=(d_1-1)!(d_2-2)! 4(\sigma_1^2(x) - 2\sigma_2(x))\sigma_2(y), \\
    C_3(x,y) &=(d_1-2)!(d_2-1)! 4\sigma_2(x) (\sigma_1^2(y) - 2\sigma_2(y)), \\
    C_4(x,y) &=(d_1-2)!(d_2-2)! 4 \sigma_2(x)\sigma_2(y),
\end{split}
\end{equation}
with the expressions for $B_h(y)$ obtained from those for $A_h(x)$ by interchanging $d_1$ and $d_2$ and replacing $x$ by $y$, and where we set $\sigma_h(x)=0$ if $h > d_1$ (respectively, $\sigma_h(y)=0$ if $h > d_2$) and $m!=0$ if $m < 0$.


The terms $P_h, Q_h$ and $S_h$ are quadratic forms in the components of $\cR$, where in particular,
\begin{equation}\label{eq:PhQhS1}
\begin{split}
    P_1 & = \sum_i p_{1|i} = \sum_{ikl} \cR_{ikil}^2 + \sum_{iab} \cR_{iaib}^2,\\
    P_3 & = \sum_{i \ne j} p_{3|ij} = \sum_{i \ne j} \Big(\sum_{kl} \cR_{ikil}\cR_{jkjl} + \sum_{kl} \cR_{ikjl}^2 + \sum_{kl} \cR_{ikjl}\cR_{jkil} + \sum_{ab} \cR_{iaib}\cR_{jajb} + 2 \sum_{ab} \cR_{iajb}^2\Big),\\
    Q_1 & =\sum_a q_{1|a} = \sum_{acd} \cR_{acad}^2 + \sum_{kla} \cR_{akal}^2,\\
    Q_3 & = \sum_{a\ne b} q_{3|ab} = \sum_{a\ne b} \Big(\sum_{cd} \cR_{acad}\cR_{bcbd} + \sum_{cd} \cR_{acbd}^2 + \sum_{cd} \cR_{acbd}\cR_{bcad} + \sum_{kl} \cR_{akal}\cR_{bkbl} + 2 \sum_{kl} \cR_{akbl}^2\Big),\\
    S_1 & = \sum_{ia} s_{1|ia} = \sum_{ikla} \cR_{ikil} \cR_{akal} + \sum_{iacd} \cR_{icid} \cR_{acad} + \sum_{ikac} \cR_{akic}^2.
\end{split}
\end{equation}

From \eqref{eq:RXsq} by \eqref{eq:AhZ} we get
\begin{equation}\label{eq:Sxrhs}
\begin{split}
    \mS(X) &= H d_1! d_2! \|X\|^4 = H d_1! d_2! (\sigma_1^2(x)- 2 \sigma_2(x)+\sigma_1^2(y)- 2 \sigma_2(y))^2 \\
    &= H (d_1 A_1(x)+ d_1(d_1-1) A_3(x) + d_2 B_1(y)+ d_2(d_2-1) B_3(y) + d_1d_2 C_1(x,y)).
\end{split}
\end{equation}

We consider two cases.

\smallskip

\underline{Case 1. $d_1=1$.} Take $X=e_1+\mathrm{i}e_2$, so that $\sigma_1(x)=1, \sigma_1(y)=\mathrm{i}, \sigma_2(x)=\sigma_2(y)=0$. Then from \eqref{eq:AhZ} we get $A_1(x) = d_2!, \, B_1(y) = (d_2-1)!, \, C_1(x,y) = - 2(d_2-1)!$, and all the other $A_h(x), B_h(y), C_h(x,y)$ are zeros. It follows from \eqref{eq:Sxrhs} that $\mS(X)=0$, and from \eqref{eq:Sx}, that $\mS(X) = (d_2-1)!(d_2 P_1 + Q_1 - 2 S_1)$, and so $d_2 P_1 + Q_1 - 2 S_1=0$. But from \eqref{eq:PhQhS1} 
\begin{align*}
    d_2 P_1 + Q_1 - 2 S_1 &= d_2 \sum_{ab} \cR_{1a1b}^2 + \sum_{acd} \cR_{acad}^2 + \sum_{a} \cR_{a1a1}^2-2\sum_{acd} \cR_{1c1d} \cR_{acad} - 2 \sum_{ac} \cR_{1a1c}^2\\
    &=\sum_{a \ne c,d}(\cR_{1c1d}-\cR_{acad})^2.
\end{align*}
Then $\cR_{1c1d}=\cR_{acad}$, for all $a, c, d$ such that $a \ne c,d$. As the choice of the orthonormal basis $\{e_a\}$ for $W_2$ is arbitrary and as $\cR_{1acd}=0$ by \eqref{eq:cR1}, the claim easily follows.

\smallskip

\underline{Case 2. $d_1 \ge 2$.} The proof is similar to that in Case 1, but we need more than one vector $X$.

Let $\cX=\{X^1, \dots, X^m\}$ be a set of vectors $X^\alpha=\sum_i x_i^\alpha e_i + \sum_a y_a^\alpha e_a \in \mathcal{Z}$. Denote $\mS(\cX) = \sum_{\alpha=1}^m \mS(X^\alpha)$ and $A_h(\cX) = \sum_{\alpha=1}^m A_h(x^\alpha), \; B_h(\cX) = \sum_{\alpha=1}^m B_h(y^\alpha), \; C_h(\cX) = \sum_{\alpha=1}^m C_h(x^\alpha, y^\alpha)$. Then by \eqref{eq:Sx}
\begin{equation} \label{eq:SX}
    \mS(\cX) = \sum_{h=1}^3 A_h(\cX) P_h + \sum_{h=1}^3 B_h(\cX) Q_h + \sum_{h=1}^4 C_h(\cX) S_h,
\end{equation}
and by \eqref{eq:Sxrhs},
\begin{equation}\label{eq:SXrhs}
    \mS(\cX) = H (d_1 A_1(\cX)+ d_1(d_1-1) A_3(\cX) + d_2 B_1(\cX)+ d_2(d_2-1) B_3(\cX) + d_1d_2 C_1(\cX)).
\end{equation}

We require the following fact.
\begin{lemma}\label{l:exists}
For arbitrary complex numbers $a_h, b_h, \; h=1,2,3$, and $c_h, \; h=1,2,3,4$, there exists a set $\cX=\{X^1, \dots, X^m\} \subset \mathcal{Z}$ such that
$A_h(\cX) = a_h, \, B_h(\cX) = b_h, \; h=1,2,3$, and $C_h(\cX) = c_h, \; h=1,2,3,4$.
\end{lemma}

\begin{proof}
The set of $10$-dimensional vectors $(A_1(\cX), A_2(\cX), A_3(\cX), B_1(\cX), B_2(\cX), B_3(\cX), C_1(\cX), C_2(\cX), C_3(\cX),$ $ C_4(\cX))$ is a linear subspace of $\mathbb{C}^{10}$. Suppose that subspace is proper. Then there exists a nontrivial linear combination $\sum_{h=1}^3 \mu_h A_h(\cX) + \sum_{h=1}^3 \nu_h B_h(\cX) + \sum_{h=1}^4 \la_h C_h(\cX)$ which vanishes for any $\cX$, and so $\sum_{h=1}^3 \mu_h A_h(x) + \sum_{h=1}^3 \nu_h B_h(y) + \sum_{h=1}^4 \la_h C_h(x,y)=0$ for any $X \in \mathcal{Z}$. But from \eqref{eq:AhZ}, taking $y=0, \sigma_1(x)=1, \sigma_2(x)=0$ we get $\mu_1=0$, then taking $y=0, \sigma_1(x)= 1,  \sigma_2(x) = \sqrt{2}$ we get $\mu_3=0$, and then taking $y=0, \sigma_1(x) = 0, \sigma_2(x) =1$ we get $\mu_2=0$. Similarly $\nu_1=\nu_2=\nu_3=0$. Then again from \eqref{eq:AhZ} taking $\sigma_1(x)= \sigma_1(y)= 1, \sigma_2(x) = \sigma_2(y) = 0$ we get $\la_1=0$, then taking $\sigma_1(x)= 1, \sigma_2(x) = 0, \sigma_2(y) = 1$ we get $\la_2=0$, then taking $\sigma_1(y)= 1, \sigma_2(y) = 0$, $\sigma_2(x) = 1$ we get $\la_3=0$, and then taking $\sigma_2(x) = \sigma_2(y) = 1$ we get $\la_4=0$.
\end{proof}

Take positive real numbers $\xi$ and $\eta$ satisfying the following inequalities:
\begin{equation}\label{eq:munu}
\begin{gathered}
    \mu:=(d_2-1)\eta+(d_2-d_1-1)\xi > 0, \qquad \nu:=(d_1-1)\xi+(d_1-d_2-1)\eta > 0, \\
    (d_1-d_2-2) \eta + d_1 \xi \ge 0, \qquad (d_2-d_1-2) \xi + d_2 \eta \ge 0
\end{gathered}
\end{equation}
(recall that $2 \le d_1 \le d_2$ and $d_1+d_2 > 4$). If $d_2=d_1$ we can take $\xi=\eta=1$; if $d_2 \ge d_1+2$ we can take $\eta=1$ and $\xi > \max((d_1-1)^{-1}(d_2-d_1+1), d_1^{-1}(d_2-d_1+2))$; if $d_2 \ge d_1+1$ we can take $\eta=1$ and $\xi \in (\max(2(d_1-1)^{-1}, 3 d_1^{-1}), d_1+1)$.

By Lemma~\ref{l:exists} we can choose $\cX$ in such a way that
\begin{gather*}
A_2(\cX) = B_2(\cX) = C_2(\cX) = C_3(\cX) = C_4(\cX) = 0, \quad A_3(\cX) = \xi, \quad  B_3(\cX) = \eta, \\
C_1(\cX) = -2 (\xi + \eta), \quad A_1(\cX) = (d_2-d_1+1) \xi + d_2 \eta, \quad B_1(\cX) = (d_1-d_2+1) \eta + d_1 \xi.
\end{gather*}
Then by \eqref{eq:SXrhs} $\mS(\cX) = 0$, and so from \eqref{eq:SX} and \eqref{eq:PhQhS1} we get
\begin{equation}\label{eq:finalqf}
\begin{split}
0 & = A_1(\cX) P_1 + A_3(\cX) P_3 + B_1(\cX) Q_1 + B_3(\cX) Q_3 + C_1(\cX) S_1 \\
& = ((d_2-d_1+1) \xi + d_2 \eta) P_1 + \xi P_3 + ((d_1-d_2+1) \eta + d_1 \xi) Q_1 + \eta Q_3 - 2(\xi + \eta) S_1 \\
& = ((d_2-d_1+1) \xi + d_2 \eta) \Big(\sum_{ikl} \cR_{ikil}^2 + \sum_{iab} \cR_{iaib}^2\Big) \\
& \quad + \xi \sum_{i \ne j} \Big(\sum_{kl} \cR_{ikil}\cR_{jkjl} + \sum_{kl} \cR_{ikjl}^2 + \sum_{kl} \cR_{ikjl}\cR_{jkil} + \sum_{ab} \cR_{iaib}\cR_{jajb} + 2 \sum_{ab} \cR_{iajb}^2\Big) \\
& \quad + ((d_1-d_2+1) \eta + d_1 \xi) \Big(\sum_{acd} \cR_{acad}^2 + \sum_{kla} \cR_{akal}^2\Big) \\
& \quad + \eta \sum_{a\ne b} \Big(\sum_{cd} \cR_{acad}\cR_{bcbd} + \sum_{cd} \cR_{acbd}^2 + \sum_{cd} \cR_{acbd}\cR_{bcad} + \sum_{kl} \cR_{akal}\cR_{bkbl} + 2 \sum_{kl} \cR_{akbl}^2\Big) \\
& \quad - 2(\xi + \eta) \Big(\sum_{ikla} \cR_{ikil} \cR_{akal} + \sum_{iacd} \cR_{icid} \cR_{acad} + \sum_{ikac} \cR_{akic}^2\Big)
\end{split}
\end{equation}
We now show that the quadratic form $\cQ$ on the right-hand side of \eqref{eq:finalqf} is positive semidefinite in the components of $\cR$. The form $\cQ$ does not contain the components $\cR_{iajb}$ with $i \ne j, \, a \ne b$, as the corresponding terms cancel out. Furthermore, we have $\cQ=\cQ_1+\cQ_2+\cQ_3+\cQ_4$, where $\cQ_1$ only involves the components $\cR_{ikjl}$ with $i, j, k, l$ pairwise non-equal; $\cQ_2$, only the components $\cR_{abcd}$ with $a, b, c, d$ pairwise non-equal; $\cQ_3$, only the components $\cR_{ijij}, \cR_{abab}, \cR_{iaia}$; and $\cQ_4$, only the components $\cR_{ijik}, \cR_{ibic}, \cR_{ajak}, \cR_{abac}$  with $j \ne k$ and $b \ne c$.

We have
\begin{equation*} 
    \cQ_1 = {\frac12 \xi \sum_{ijkl}}'(\cR_{ikjl} + \cR_{jkil})^2, \qquad  \cQ_2 = {\frac12 \eta \sum_{acbd}}'(\cR_{acbd} + \cR_{adbc})^2.
\end{equation*}
To simplify the form $\cQ_3$ we denote $U_k = \sum_i \cR_{ikik}, V_k = \sum_a \cR_{akak}, U_a = \sum_c \cR_{caca}, V_a = \sum_i \cR_{iaia}$. Then
\begin{align*}
\cQ_3 &=
((d_2-d_1+1) \xi + d_2 \eta) \sum_{ik} \cR_{ikik}^2 + \xi \sum_{ijk} \cR_{ikik}\cR_{jkjk} + ((d_1-d_2+1) \eta + d_1 \xi) \sum_{ac} \cR_{acac}^2 + \eta \sum_{abc} \cR_{acac}\cR_{bcbc} \\
& \quad + ((d_2-2) \xi + (d_1-2) \eta) \sum_{ia} \cR_{iaia}^2  + \xi \sum_{ija} \cR_{iaia}\cR_{jaja} + \eta \sum_{abk} \cR_{akak}\cR_{bkbk} \\
& \quad - 2(\xi + \eta) \Big(\sum_{ika} \cR_{ikik} \cR_{akak} + \sum_{iac} \cR_{icic} \cR_{acac}\Big) \\
& = ((d_2-d_1+1) \xi + d_2 \eta) \sum_{ik} \cR_{ikik}^2 + \xi \sum_{k} U_{k}^2 + ((d_1-d_2+1) \eta + d_1 \xi) \sum_{ac} \cR_{acac}^2 + \eta \sum_{c} U_{c}^2 \\
& \quad + ((d_2-2) \xi + (d_1-2) \eta) \sum_{ia} \cR_{iaia}^2  + \xi \sum_{a} V_{a}^2 + \eta \sum_{k} V_{k}^2 - 2(\xi + \eta) \Big(\sum_{k} U_{k} V_{k} + \sum_{c} U_{c} V_{c}\Big).
\end{align*}
Using the fact that $\sum_{ik} \cR_{ikik}^2= \frac{1}{2(d_1-1)} \sum_{kij}(\cR_{ikik}-\cR_{jkjk})^2+ \frac{1}{d_1-1} \sum_k U_k^2, \; \sum_{ac} \cR_{acac}^2= \frac{1}{2(d_2-1)} \sum_{acd}(\cR_{acac}-\cR_{adad})^2+ \frac{1}{d_2-1} \sum_a U_a^2$, and $((d_2-2) \xi + (d_1-2) \eta) \sum_{ia} \cR_{iaia}^2 = (\mu+\nu) \sum_{ia} \cR_{iaia}^2 = \mu (\frac{1}{2d_1} \sum_{aij}(\cR_{iaia}-\cR_{jaja})^2+ \frac{1}{d_1} \sum_a V_a^2)+ \nu (\frac{1}{2d_2} \sum_{abi}(\cR_{iaia}-\cR_{ibib})^2+ \frac{1}{d_2} \sum_i V_i^2)$, where $\mu$ and $\nu$ are given by \eqref{eq:munu}, we obtain
\begin{equation}\label{eq:Q3}
\begin{split}
\cQ_3 &= \frac{(d_2-d_1+1) \xi + d_2 \eta}{2(d_1-1)} \sum_{kij}(\cR_{ikik}-\cR_{jkjk})^2
+ \frac{(d_1-d_2+1) \eta + d_1 \xi}{2(d_2-1)} \sum_{acd}(\cR_{acac}-\cR_{adad})^2 \\
& \quad + \frac{\mu}{2d_1} \sum_{aij}(\cR_{iaia}-\cR_{jaja})^2
+ \frac{\nu}{2d_2} \sum_{abi}(\cR_{iaia}-\cR_{ibib})^2\\
& \quad + (\xi + \eta)d_2(d_1-1)\sum_k \Big(\frac{1}{d_1-1}U_k - \frac{1}{d_2} V_{k}\Big)^2 + (\xi + \eta)d_1(d_2-1) \sum_a \Big(\frac{1}{d_2-1}U_a-\frac{1}{d_1}V_a\Big)^2.
\end{split}
\end{equation}
We now consider the quadratic form $\cQ_4$. Collecting the terms we obtain
\begin{align*}
\cQ_4 &= \sum_{k \ne l} \Big[((d_2-d_1+4) \xi + d_2 \eta) \sum_{i} \cR_{ikil}^2 + ((d_1-d_2-2) \eta + d_1 \xi) \sum_{a} \cR_{akal}^2  \\
& \qquad \qquad + \xi \sum_{ij} \cR_{ikil}\cR_{jkjl} + \eta \sum_{ab} \cR_{akal}\cR_{bkbl} - 2(\xi + \eta) \sum_{ia} \cR_{ikil} \cR_{akal}\Big] \\
& + \sum_{c \ne d} \Big[ ((d_1-d_2+4) \eta + d_1 \xi) \sum_a \cR_{acad}^2  + ((d_2-d_1-2) \xi + d_2 \eta) \sum_{i} \cR_{icid}^2  \\
& \qquad \qquad + \eta \sum_{ab} \cR_{acad}\cR_{bcbd} + \xi \sum_{ij} \cR_{icid}\cR_{jcjd} - 2(\xi + \eta) \sum_{c \ne d;ia} \cR_{icid} \cR_{acad} \Big].
\end{align*}
The form $\cQ_4$ is positive semidefinite if the expressions in the square brackets are non-negative. For the first one, it suffices to show that the quadratic form $f=((d_2-d_1+4) \xi + d_2 \eta) \sum_{i} u_{i}^2 + ((d_1-d_2-2) \eta + d_1 \xi) \sum_{a} v_{a}^2 + \xi \sum_{ij} u_{i} u_{j} + \eta \sum_{ab} v_a v_b - 2(\xi + \eta) \sum_{ia} u_i v_a$ in the variables $u_1, \dots, u_{d_1-2}, v_1, \dots, v_{d_2}$ is positive semidefinite. Note that $\xi, \eta > 0$ and that $(d_2-d_1+4) \xi + d_2 \eta, (d_1-d_2-2) \eta + d_1 \xi \ge 0$ by \eqref{eq:munu}. If $d_1=2$, there is no $u_i$'s and $f \ge 0$ trivially. If $d_1 > 2$, then by an orthogonal change of variables $\{u_i\} \mapsto \{u_i'\}, \; \{v_a\} \mapsto \{v_a'\}$ such that $u_1'=(d_1-2)^{-1/2} \sum_{i} u_{i}, \; v_1'= d_2^{-1/2} \sum_{a} v_{a}$, we get $f=((d_2+2) \xi + d_2 \eta) {u_1'}^2 + ((d_1-2) \eta + d_1 \xi) {v_{1}'}^2 - 2(\xi + \eta) \sqrt{d_2(d_1-2)} u_1' v_1' + ((d_2-d_1+4) \xi + d_2 \eta) \sum_{i>1} {u_{i}'}^2 + ((d_1-d_2-2) \eta + d_1 \xi) \sum_{a>1} {v_{a}'}^2$, and so $f \ge 0$, as $((d_2+2) \xi + d_2 \eta)((d_1-2) \eta + d_1 \xi) -d_2(d_1-2)(\xi + \eta)^2 = 2(d_1+d_2)\xi^2+2(d_1+d_2-2) \xi \eta > 0$. A similar argument for the second square bracket shows that $\cQ_4 \ge 0$.

It now follows from \eqref{eq:finalqf} that $\cQ_1=\cQ_2=\cQ_3=0$, which by \eqref{eq:Q3} implies that $\cR_{ikik}=\cR_{jkjk}, \cR_{acac}=\cR_{adad}$, $\cR_{iaia}=\cR_{jaja}, \cR_{iaia}=\cR_{ibib}$, and $d_2 U_k =(d_1-1) V_{k}, \; d_1 U_a = (d_2-1) V_a$. As the choice of the bases for $W_1, W_2$ is arbitrary, the claim follows.
\end{proof}


%


\end{document}